\documentclass[12pt]{amsart}

\usepackage{amssymb,amsfonts,amsthm}
\usepackage{graphicx}

\newtheorem{theorem}{Theorem}[section]
\newtheorem{lemma}[theorem]{Lemma}
\newtheorem{corollary}[theorem]{Corollary}

\theoremstyle{definition}

\newtheorem{example}[theorem]{Example}
\newtheorem{remark}[theorem]{Remark}

\numberwithin{equation}{section}

\newcounter{minutes}
\divide\time by 60
\newcounter{hours}
\multiply\time by 60 \addtocounter{minutes}{-\time}

\begin{document}

\title[Nonlinear integral transforms]
{Preserving properties and pre-Schwarzian norms of nonlinear integral transforms}

\author[Shankey Kumar]{Shankey Kumar}
\author[S.K. Sahoo]{S.K. Sahoo$^*$}

\thanks{The authors would like to thank Professor Toshiyuki Sugawa for his useful remarks leading 
to some improvement in the paper.
The work of the first author is supported by CSIR, New Delhi (Grant No: 09/1022(0034)/2017-EMR-I)}


\keywords{Integral transform, Hornich operators, Ces\`{a}ro transform, 
Pre-Schwarzian norm, Univalent functions, Spirallike functions, Convex functions, 
Close-to-Convex functions\\
${}^{\mathbf{*}}$ {\tt Corresponding author}}

\maketitle

\begin{center}
Discipline of Mathematics,
Indian Institute of Technology Indore,
Indore 453552, India\\
email: {\tt shankeygarg93@gmail.com}\\
email: {\tt swadesh.sahoo@iiti.ac.in}
\end{center}

\begin{abstract}
In this article, we study preserving properties of certain nonlinear 
integral transforms in some classical families of normalized analytic 
univalent functions defined in the unit disk. 
Also, we find sharp pre-Schwarzian norm estimates of such integrals.
\end{abstract}

\section{Introduction and Preliminaries}
Let $\mathcal{A}$ denote the class of all analytic functions $f$ in the 
unit disk $\mathbb{D}:=\{z\in\mathbb{C}:|z|<1\}$ 
with the normalization $f(0)=0$ and $f'(0)=1$. The subclass of $\mathcal{A}$ 
consisting of all univalent functions is denoted by $\mathcal{S}$. 
The notations $\mathcal{S}^*$ and $\mathcal{K}$ stand for the well-known classes 
of functions in $\mathcal{S}$ that are 
starlike (with respect to origin) and convex, respectively, see \cite{Duren83}. 
We also denote by $\mathcal{C}$, the class of close-to-convex functions in $\mathbb{D}$, 
i.e. functions $f\in \mathcal{A}$ satisfying
$$
{\rm Re}\,\bigg(e^{i\alpha}\frac{f'(z)}{g'(z)}\bigg)>0
$$
for some $g\in\mathcal{K}$ and a real number $\alpha\in(-\pi/2,\pi/2)$ 
(see \cite[Vol. 2, p.~2]{GMBook}).
Some natural generalizations of the classes $\mathcal{S}^*$ and $\mathcal{K}$ are 
available in the literature. In this paper, we consider the following generalizations:
\begin{equation}\label{2eq1.1}
\mathcal{S}^*_\alpha(\lambda)=\Bigg\lbrace f\in \mathcal{A}:{\rm Re}\,
\bigg(e^{i\alpha}\frac{zf'(z)}{f(z)}\bigg)
    >\lambda \cos \alpha\Bigg\rbrace,
\end{equation}
and
\begin{equation}\label{2eq1.2}
\mathcal{K}(\lambda)=\Bigg\lbrace f\in \mathcal{A}:{\rm Re}\,\bigg(\frac{zf''(z)}{f'(z)}+1\bigg)
   >\lambda \Bigg\rbrace,
\end{equation}
where $\alpha \in (-\pi/2,\pi/2)$ and $\lambda<1$. Note that the class 
$\mathcal{S}^*_\alpha(\lambda)$ is known as the class of $\alpha$-spirallike 
functions of order $\lambda$ and the class $\mathcal{K}(\lambda)$ denotes the 
class of convex functions of order $\lambda$. Recall that 
the class $\mathcal{S}^*_\alpha(\lambda)$, for $0\le \lambda<1$, is studied by 
several authors in difference perspective (see, for instance, \cite[p~93, Vol.~2]{GMBook} 
and \cite{Lib67,PW14}). Further, the class
$\mathcal{K}(\lambda)$, $-1/2\le \lambda<1$, is introduced, for instance, in \cite{LPQ16} 
and references therein. 
Originally, a slight modification of this class was first studied by T. Umezawa in 1952 \cite{Ume52}
by characterizing with the class of functions convex in one direction.
We can also easily observe that the class 
$\mathcal{K}(\lambda)$, $-1/2\le\lambda<1$, is
contained in the class $\mathcal{C}$ that follows from Kaplan's Theorem, see \cite[$\S$2.6]{Duren83}.
Note that $\mathcal{K}(\lambda)$, $0\le \lambda<1$, is the well-known 
class of normalized convex univalent functions.

Recall from the literature that 
$$
\mathcal{S}^*(\lambda):=\mathcal{S}^*_0(\lambda),~
\mathcal{S}^*:=\mathcal{S}^*(0)
~~\mbox{ and }~~
\mathcal{K}:=\mathcal{K}(0).
$$
Motivation to consider the class $\mathcal{S}^*(\lambda)$, $ \lambda<1$, comes, 
for instance, from the classes $\mathcal{S}^*(-1/2)$
and $\mathcal{K}(-1/2)$ already studied in the literature (see \cite[p.~66]{MM-Book} 
for some interesting results).
An interesting relation between the classes $\mathcal{S}^*$ and $\mathcal{K}$ is the 
classical Alexander Theorem which states that $f\in \mathcal{S}^*$ if and only if 
$J[f]\in \mathcal{K}$, where $J[f]$ denotes the Alexander transform of $f\in\mathcal{A}$ defined as  
$$
J[f](z)=\int_{0}^{z} \frac{f(w)}{w}dw.
$$
This operator is one of the main operators we consider in this paper.
We know that the class $\mathcal{S}$ does not preserve by the Alexander transform, 
see \cite[$\S8.4$]{Duren83}. This motivates us to study 
the classical classes of functions whose 
images lie on the class $\mathcal{S}$ under
the Alexander and related transforms considered 
in this paper.
We use the following notation concerning the Alexander operator $J$: 
\begin{equation}\label{JF}
J(\mathfrak{F})=\{J[f]:\,f\in\mathfrak{F}\}
\end{equation}
with $\mathfrak{F}:=\{f\in\mathcal{A}:f'(z)\neq 0\}$, the class of 
{\em locally univalent functions}. We say that a function $g\in J(\mathfrak{F})$ 
if and only if $g=J[f]$ for some $f\in\mathfrak{F}$.

The second integral operator that we study in this paper is 
$I_\gamma$, the {\em Hornich scalar multiplication operator} of $f\in\mathcal{A}$ defined by
\begin{equation}\label{HSOper}
I_\gamma[f](z)=(\gamma\star f)(z)=\int_{0}^{z}\{f'(w)\}^\gamma dw,
\end{equation}
where the branch of $\{f'(w)\}^\gamma=\exp(\gamma\log f'(z))$ is chosen so that 
$\{f'(0)\}^\gamma=1$. It clearly follows that $I_\alpha I_\beta = I_{\alpha \beta}$. 
In the sequel, the following definition due to Y. J. Kim and E. P. Merkes \cite{Kim72} 
is useful for our main results:
\begin{equation}\label{AF}
A(\mathfrak{F})=\{\gamma \in \mathbb{C}: I_\gamma(\mathfrak{F})\subset \mathcal{S}\}
\end{equation}
with $\mathfrak{F}$ as defined above. Here, the notation $I_\gamma(\mathfrak{F})$ is defined by
\begin{equation}\label{IgammaF}
I_\gamma(\mathfrak{F})=\{I_\gamma[f]:\,f\in \mathfrak{F}\}.
\end{equation}
We say that a function $g\in I_\gamma(\mathfrak{F})$ if and only if $g=I_\gamma[f]$ 
for some $f\in\mathfrak{F}$.

For simplicity in our further discussion, we introduce the notation 
$\mathcal{S}^*_\alpha:=\mathcal{S}^*_\alpha(0)$.
In \cite{Kim07}, Y. C. Kim and T. Sugawa find a condition on $\alpha$ such that 
$J[f],~f\in \mathcal{S}^*_\alpha$, is univalent with the help of the problem of 
determining the set $A(J(\mathcal{S}^*_\alpha))$,
where $J(\mathcal{S}^*_\alpha)$ is defined similar to the definition \eqref{JF}. 
Recall that the inclusion $\{\gamma:\,|\gamma|\le 1/2\}\subset A(\mathcal{K})$ was 
first proved by 
V. Singh and P. N. Chichra in \cite{Singh77} (see also \cite{Kimponnusamy04} and \cite{Merkes85}). 
Further, the inclusion $[0,3/2]\subset A(\mathcal{K})$ was due to M. Nunokawa \cite{Nunokawa69}. 
In continuation to this analysis, in 1985, E. P. Merkes proposed the conjecture that 
$\{\gamma\in\mathbb{C}:\,|\gamma-1|\le 1/2\}\subset A(\mathcal{K})$.
However, L. A. Aksent'ev and I. R. Nezhmetdinov
\cite{Aksent'ev82}
disproved the conjecture of E. P. Merkes by showing that 
\begin{equation}\label{AK}
A(\mathcal{K})=\{\gamma\in\mathbb{C}:\,|\gamma|\le 1/2\}\cup[1/2,3/2]
\end{equation}
(see also \cite{Kimponnusamy04}).

Next we observe that 
$$
(I_\gamma \circ J)[f](z)
=\int_{0}^{z} \bigg(\frac{f(w)}{w}\bigg)^\gamma dw=: J_\gamma[f](z).
$$
It is here appropriate to notice that $J_1[f]=J[f]$.
Then by the definitions \eqref{AF} and \eqref{IgammaF} we formulate
$$
A\big(J(\mathfrak{F})\big)=\{\gamma \in \mathbb{C}: J_\gamma(\mathfrak{F})\subset \mathcal{S}\}
~~\mbox{ and }~~
J_\gamma(\mathfrak{F})=(I_\gamma\circ J)(\mathfrak{F}).
$$ 
The operator $J_\gamma[f]$ was initially considered by Y. J. Kim and E. P. Merkes 
in \cite{Kim72}, and they showed that $J_\gamma(\mathcal{S})\subset \mathcal{S}$ 
for $|\gamma|\leq 1/4$, i.e. $A(J(\mathcal{S}))=\{\gamma\in\mathbb{C}:\,|\gamma|\le 1/4\}$. 
For the starlike family $\mathcal{S}^*$, V. Singh and P. N. Chichra in \cite{Singh77} 
proved that $A(J(\mathcal{S}^*))\supset \{\gamma\in\mathbb{C}:\,|\gamma|\leq 1/2\}$. 
However, as noted in \eqref{AK}, the complete range of $\gamma$ for $A(J(\mathcal{S}^*))$ 
was found by L. A. Aksent'ev and I. R. Nezhmetdinov, since 
$J(\mathcal{S}^*)=\mathcal{K}$. More interestingly, for a given $\alpha>0$,
Y. C. Kim, S. Ponnusamy, and T. Sugawa \cite{Kim04} could generate a subclass $\mathcal{F}$ 
of $\mathcal{A}$ such that   
$J_\gamma(\mathcal{F})\subset \mathcal{S}$ for all $\gamma\in\mathbb{C}$ with $|\gamma|\leq\alpha$.

Next we are interested to investigate the univalency and preservation property of certain 
classes of functions under the $\beta$-Ces\`{a}ro transform defined by
$$
C_\beta[f](z)=\int_{0}^{z}\frac{f(w)}{w(1-w)^{\beta}}dw, \quad\text{ for } \beta \geq 0,
$$
where $f$ is analytic in $\mathbb{D}$ and $f(0)=0$. Throughout this paper we consider 
$\beta\geq0$ unless it is specified. One can express the  $\beta$-Ces\`{a}ro transform in 
terms of the Hornich operations, i.e. of the form
$$
C_\beta[f](z)=\big(J[f]\oplus (\beta \star g)\big)(z),
$$
where $g(z)=-\log(1-z)\in\mathcal{K}$. Here, the symbol $\star$ is the Hornich scalar 
multiplication operator as in \eqref{HSOper}, whereas the symbol $\oplus$ denotes the 
Hornich addition operator defined by
$$
(f\oplus g)(z)= \int_{0}^{z} f'(w)g'(w)dw,
$$
for $f,g \in \mathfrak{F}$. The Hornich operations are widely used in the literature, 
see \cite{Ali18} and \cite{ponnusamy19}. Note that the $\beta$-Ces\`{a}ro transform 
reduces to the Alexander transform if we choose $\beta=0$ and to the Ces\`{a}ro 
transform \cite{Hartmann74} if we choose $\beta=1$. We use the notation $C[f]:=C_1[f]$ 
for the Ces\`{a}ro transform. For more information about the $\beta$-Ces\`{a}ro transform, 
see \cite{Shankey18}. Here it is appropriate to recall that F. W. Hartmann and 
T. H. MacGregor in the same paper \cite{Hartmann74}  provided examples of a univalent 
function and a starlike function whose images are not univalent and starlike, respectively, 
under the Ces\`{a}ro transform. 
Recently, S. Ponnusamy, S. K. Sahoo and T. Sugawa \cite{ponnusamy19} studied the 
univalency of the Ces\`{a}ro transform and even more general transforms of functions of 
bounded boundary rotations.
 
We organize the structure of our paper as follows:
throughout the paper we assume $\alpha \in (-\pi/2,\pi/2)$ and $\lambda<1$.
First we study the univalency of the Hornich scalar multiplication operator on the class 
$\mathcal{K}(\lambda)$. By
setting $\mathcal{S}(\lambda):=\bigcup_\alpha\,\mathcal{S}^*_\alpha(\lambda)$, we next 
compute the sets $A\big(J(\mathcal{S}_\alpha^*(\lambda))\big)$ and 
$A\big(J(\mathcal{S}(\lambda))\big)$.  Also, we find the values of $\beta$ for which 
$C_\beta(\mathcal{S}^*(\lambda))=\{C_\beta[f]:f\in\mathcal{S}^*(\lambda)\}\subset \mathcal{S}$, 
$C_\beta(\mathcal{S}^*(\lambda))\subset \mathcal{S}^*$,   
$C_\beta(\mathcal{K})=\{C_\beta[f]:f\in\mathcal{K}\}\subset \mathcal{K}$ and 
$C_\beta(\mathcal{C})=\{C_\beta[f]:f\in\mathcal{C}\}\subset \mathcal{C}$. We set 
$C(\mathcal{S}^*(\lambda))=C_1(\mathcal{S}^*(\lambda))$ when we talk about the 
classical Ces\`aro transform $C[f]$.
In this context, we also have an example of univalent function whose image is not 
univalent under the $\beta$-Ces\`{a}ro transform. 
Finally, we deal with pre-Schwarzian norm of some of the above integral transforms 
and as a result we could find an alternate way to show that the class  
$\mathcal{S}^*(\lambda)$ is not contained in $\mathcal{S}$ for $\lambda<0$.

\section{Preserving Properties}
It is here appropriate to recall that,
in one hand, due to J. A. Pfaltzgraff as shown in \cite[Corollary~1]{Pfa75} 
$I_\gamma(\mathcal{S})\subset \mathcal{S}$ for $|\gamma|\le 1/4$. On the other hand, 
W. C. Royster proved in \cite[Theorem~2]{Royster65} that for each number $\gamma\neq 1$ 
with $|\gamma|>1/3$,
there exists a function $f\in\mathcal{S}$ such that $I_\gamma[f]\not\in\mathcal{S}$ 
(see also \cite{Ali18,Kim04,KS06}). 
Also, recall from \eqref{AK} that 
$A(\mathcal{K})=\{\gamma\in\mathbb{C}:\,|\gamma|\le 1/2\}\cup[1/2,3/2]$.
However, as a result of our first main result stated below which generalizes the set 
$A(\mathcal{K})$ to the set $A(\mathcal{K}(\lambda))$, $\lambda<1$, we have a larger 
class of functions $\mathcal{K}(-1/2)$ than 
$\mathcal{K}$ for which
$$
A(\mathcal{K}(-1/2))=\left\{\gamma\in\mathbb{C}:\,|\gamma|
\le \frac{1}{3}\right\}\bigcup\Big[\frac{1}{3},1\Big].
$$
Note that the description of the whole set $A(\mathcal{S})$ is still open.
 
\begin{theorem}\label{2theorem2.1}
Let $\lambda<1$. Then we have 
$$
A(\mathcal{K(\lambda)})=\left\{\gamma\in\mathbb{C}:\,|\gamma|
\le \frac{1}{2(1-\lambda)}\right\}\bigcup \left[\frac{1}{2(1-\lambda)},\frac{3}{2(1-\lambda)}\right].
$$
\end{theorem}
\begin{proof}
As observed in \cite{Koe85}, $\mathcal{K}(\lambda)$ can be expanded in terms of 
the Hornich scalar multiplication: 
$(1-\lambda)\star \mathcal{K}=\{(1-\lambda)\star f : f\in \mathcal{K} \}$.
Then, for $f\in \mathcal{K}(\lambda)$, there exists a function $g \in \mathcal{K}$ 
such that $f(z)=((1-\lambda)\star g)(z)$.  This relation gives that 
$I_\gamma[f]=I_{(1-\lambda)\gamma}[g]$ for a function $g\in\mathcal{K}$. 
This concludes the proof by the help of the set $A(\mathcal{K})$.
\end{proof}

We now collect an important lemma, which is a generalization of a result of Y. C. Kim 
and T. Sugawa (see \cite[Lemma~4]{Kim07}), to conclude our next main result and its consequences.

\begin{lemma}\label{2lemma2.2}
For $-\pi/2<\alpha <\pi/2$ and $\lambda<1$, we have
$$
J(\mathcal{S}_{\alpha}^*(\lambda))=I_{e^{-i\alpha}\cos\alpha}(\mathcal{K}(\lambda)).
$$
\end{lemma}
\begin{proof}
Let $f\in J(\mathcal{S}_{\alpha}^*(\lambda))$. We write
$$
\frac{1}{\cos\alpha}\Bigg[e^{i\alpha}\bigg(\frac{zf''(z)}{f'(z)}+1\bigg)-i \sin\alpha\Bigg] =p(z),
$$
where $p$ is an analytic function in $|z|<1$. Clearly, $p(0)=1$ and ${\rm Re}\, p(z)>\lambda$.

If we take $k\in \mathcal{K}(\lambda)$ such that $1+zk''(z)/k'(z)=p(z)$ then we obtain
$$
\frac{f''(z)}{f'(z)}=e^{-i\alpha}\cos\alpha \frac{k''(z)}{k'(z)},
$$
which yields $f=I_{e^{-i\alpha}\cos \alpha}[k]$. This follows that 
$J(\mathcal{S}_{\alpha}^*(\lambda))\subset I_{e^{-i\alpha}\cos\alpha}(\mathcal{K}(\lambda))$. 
If we take the backward process, then we obtain the reverse inclusion 
$J(\mathcal{S}_{\alpha}^*(\lambda))\supset I_{e^{-i\alpha}\cos\alpha}(\mathcal{K}(\lambda))$. 
The desired result is thus obtained.
\end{proof}

For $z,w\in\mathbb{C}$, we denote by $[z,w]$ for
the line segment joining $z$ and $w$.
An immediate consequence of Theorem~\ref{2theorem2.1} and Lemma \ref{2lemma2.2} leads to 
the following theorem:

\begin{theorem}\label{2theorem2.3}
For $-\pi/2<\alpha<\pi/2$ and $\lambda<1$, we have
$$
A\big(J(\mathcal{S}_{\alpha}^*(\lambda))\big)=\bigg\{\gamma\in\mathbb{C}:\, |\gamma|
\leq\frac{1}{2(1-\lambda)\cos\alpha}\bigg\} \bigcup \bigg[\frac{e^{i\alpha}}{2(1-\lambda)\cos\alpha},
\frac{3e^{i\alpha}}{2(1-\lambda)\cos\alpha}\bigg].
$$	
\end{theorem}
\begin{proof} By using Lemma \ref{2lemma2.2} and the property $I_a I_b= I_{ab}$, 
for $a,b\in\mathbb{C}$, we have
$$
I_\gamma\big(J(\mathcal{S}_{\alpha}^*(\lambda))\big)
=I_\gamma I_{e^{-i\alpha}\cos\alpha}(\mathcal{K}(\lambda))
=I_{\gamma e^{-i\alpha}\cos\alpha}(\mathcal{K}(\lambda)).
$$
Therefore, $\gamma\in A\big(J(\mathcal{S}_{\alpha}^*(\lambda))\big)$ 
if and only if $\gamma e^{-i\alpha}\cos\alpha \in A\big(\mathcal{K}(\lambda)\big)$. 
Now we are able to conclude the proof by Theorem \ref{2theorem2.1}.
\end{proof}

We remark that the special choice $\lambda=0$ takes Theorem~\ref{2theorem2.3} 
to \cite[Theorem~3]{Kim07}.

By the definition of $\mathcal{S}(\lambda)$, we have 
$$
A\big(J(\mathcal{S}(\lambda))\big)=\underset{\alpha}{\bigcap}\,
A\big(J\big(\mathcal{S}_{\alpha}^*(\lambda)\big)\big).
$$
Using Theorem \ref{2theorem2.3}, we now conclude the following theorem. 

\begin{theorem}
For $\lambda<1$, we have
$$
A\big(J(\mathcal{S}(\lambda))\big)=\bigg\lbrace |\gamma|\leq\frac{1}{2(1-\lambda)}\bigg\rbrace.
$$	
\end{theorem}
This theorem for the special case $\lambda=0$ was considered in \cite{Kim07}. 

In the next theorem, we have the inclusion of the image set 
$J(\mathcal{S}_{\alpha}^*(\lambda))$ in the class $\mathcal{S}$ for some restrictions on
$\alpha$. However, the case $\lambda=0$ has also been considered in \cite{Kim07}. 

\begin{theorem}
If $\lambda<1$,
then the relation
$$J(\mathcal{S}_{\alpha}^*(\lambda))\subset \mathcal{S}$$
holds precisely for $\cos\alpha\leq1/2(1-\lambda)$. However, if $-1/2\leq \lambda<1$, 
then the same inclusion follows for $\alpha=0$.	
\end{theorem}
\begin{proof}
If $\alpha=0$, the result is trivial to prove.
Indeed, in this case, we have $J[f]\in \mathcal{C}\subset\mathcal{S}$ for 
$f\in \mathcal{S}^*(\lambda)$, $-1/2\leq\lambda<1$.

Thus, we assume that $\alpha\neq 0$. We have $J\big(\mathcal{S}_{\alpha}^*(\lambda)\big)
\subset \mathcal{S}$ if and only if $1\in A\big(J\big(\mathcal{S}_{\alpha}^*(\lambda))\big)$. 
This gives that $\cos \alpha \leq 1/2(1-\lambda\big)$, completing the proof.
\end{proof}

The following lemma gives a relation of the $\beta$-Ces\`aro transform of with the 
transform $J_\gamma$ for $\gamma=e^{-i\alpha}\sec\alpha$.

\begin{lemma}\label{2lemma2.6} 
Let $\beta\ge 0$ and  $-\pi/2<\alpha<\pi/2$.
Let $f\in\mathcal{A}$ be such that 
$$
\Bigg[\frac{g(z)}{z(1-z)^\beta} \Bigg]^{e^{i\alpha}\cos\alpha }=\frac{f(z)}{z}
$$
for some $g\in \mathcal{S}^*(\lambda)$, then 
$f\in \mathcal{S}_{\alpha}^*(\lambda-\beta/2)$
for $\lambda<1$.
\end{lemma}
\begin{proof}
Suppose that $g\in \mathcal{S}^*(\lambda)$ and
$$
\Bigg[\frac{g(z)}{z(1-z)^\beta} \Bigg]^{e^{i\alpha}\cos\alpha }=\frac{f(z)}{z}.
$$
Logarithm derivative obtains
$$
e^{i\alpha}\Bigg[\frac{zf'(z)}{f(z)}-1\Bigg]
=\cos \alpha \Bigg[\frac{zg'(z)}{g(z)}-1+\frac{\beta z}{1-z}\Bigg],
$$
which implies that 
$$
{\rm Re}\,\Bigg[e^{i\alpha}\frac{zf'(z)}{f(z)}\Bigg]
={\rm Re}\, \Bigg[\cos \alpha\frac{zg'(z)}{g(z)}\Bigg]+{\rm Re}\,
\Bigg[\cos\alpha\frac{\beta z}{1-z}\Bigg].
$$
Since $g\in \mathcal{S}^*(\lambda)$ and ${\rm Re}\,(z/(1-z))>-1/2$ for $|z|<1$, it follows that
$$
{\rm Re}\,\Bigg[e^{i\alpha}\frac{zf'(z)}{f(z)}\Bigg]
>\bigg(\lambda-\frac{\beta}{2}\bigg)\cos\alpha.
$$
Thus, $f\in \mathcal{S}_{\alpha}^*(\lambda-\beta/2)$
for $\lambda<1$ follows by the definition \eqref{2eq1.1}, completing the proof.
\end{proof}

As an application of Theorem~\ref{2theorem2.3} and Lemma~\ref{2lemma2.6}, 
we next find restriction on $\beta$ for which $C_\beta(\mathcal{S}^*(\lambda))$ 
is contained in $\mathcal{S}$. 

\begin{theorem}\label{2theorem2.7}
For $-1/2\leq\lambda<1$ and $0\leq\beta\leq 2\lambda+1$, the relation 
$C_\beta(\mathcal{S}^*(\lambda))\subset \mathcal{S}$ holds.
\end{theorem}
\begin{proof}
Substituting $\alpha=0$ in Lemma \ref{2lemma2.6}, for a given function 
$g\in \mathcal{S}^*(\lambda)$, we can find another function 
$f \in \mathcal{S}^*(\lambda-\beta/2)$ satisfying
$$
\int_0^z \frac{g(w)}{w(1-w)^\beta}dw =\int_0^z \frac{f(w)}{w}dw.
$$
Secondly, Theorem \ref{2theorem2.3} gives that $J_{\gamma}(\mathcal{S}^*(\lambda-\beta/2))
\subset \mathcal{S}$ whenever $\gamma$ lies either in 
$\lbrace \gamma\in\mathbb{C}:\, |\gamma|\leq1/2(1-\lambda+\beta/2)\rbrace$ or in 
$[1/2(1-\lambda+\beta/2),3/2(1-\lambda+\beta/2)]$. It follows that 
$$
C_\beta(\mathcal{S}^*(\lambda))\subset \mathcal{S} 
\quad \mbox{for $1 \leq \dfrac{3}{2(1-\lambda+\beta/2)}$},
$$
that is for $\beta\leq2\lambda+1$. This completes the proof.
\end{proof}

We remark that Theorem~\ref{2theorem2.7} can be proved alternatively by using 
the classical theorem of Kaplan (\cite[$\S$2.6]{Duren83}) which states that
$f\in \mathcal{C}$ if and only if
$$
\int_{\theta_1}^{\theta_2}{\rm Re}\,\bigg(1+\frac{zf''(z)}{f'(z)}\bigg)d\theta>-\pi,
$$
whenever $0\leq\theta_1<\theta_2\leq 2\pi$.
Using this, we obtain
\begin{align*}
\int_{\theta_1}^{\theta_2}{\rm Re}\,\bigg(1+\frac{zC_\beta[f]''(z)}
{C_\beta[f]'(z)}\bigg)d\theta &= \int_{\theta_1}^{\theta_2}{\rm Re}\,
\bigg(\frac{zf'(z)}{f(z)}+\frac{\beta z}{1-z}\bigg)d\theta\\
&>\lambda(\theta_2-\theta_1)-\frac{\beta}{2}(\theta_2-\theta_1)\geq -(\beta-2\lambda) \pi.
\end{align*}
This gives that $C_\beta[f]\in \mathcal{C}\subset\mathcal{S}$, for $\beta\leq 2\lambda+1$. 

In the following example, we show that the quantity $2\lambda+1$ can not be replaced 
by any bigger number in Theorem \ref{2theorem2.7}.

\begin{example}\label{2example2.8}
For $-1/2\leq\lambda<1$, let $f(z)=z/(1-z)^{2-2\lambda}$. Recall that this $f$ is an 
element of the class $\mathcal{S}^*(\lambda)$. From the definition of $C_\beta[f]$ we obtain
\begin{equation}\label{C-beta}
C_\beta[f](z)=\int_{0}^{z}\frac{1}{(1-w)^{\beta-2\lambda+2}}dw=\frac{1}{(\beta-2\lambda+1)}
\Bigg[\frac{1}{(1-z)^{(\beta-2\lambda+1)}} - 1\Bigg].
\end{equation}
Note that $C_\beta[f]$ is univalent if and only if 
$g(z)=(1-z)^{1-\beta-2+2\lambda}$ is univalent. However, by the lemma of 
W. C. Royster stated in \cite[p.~386]{Royster65}, we obtain that $g(z)$ is univalent if and only if 
$2\lambda-3\leq\beta\leq 2\lambda+1$. It follows that if $\beta>2\lambda+1$, 
then $C_\beta[f]$ does not lie on the class $\mathcal{S}$.
\end{example}

If we choose $\beta=1$ in Theorem \ref{2theorem2.7}, it produces the following 
well-known result \cite{Hartmann74} concerning the Ces\`{a}ro transform $C[f]$ on 
the class $\mathcal{S}^*(\lambda)$, $0\leq\lambda<1$:

\begin{corollary}
For $0\leq\lambda<1$, the relation $C(\mathcal{S}^*(\lambda))\subset \mathcal{S}$ holds.
\end{corollary}

In the statement of Theorem \ref{2theorem2.7}, if we replace $\mathcal{S}$ 
by $\mathcal{S}^*$ then we have new restriction on $\beta$, which is described in the following theorem:

\begin{theorem}
For $0\leq\lambda<1$ and $0\leq\beta\leq2\lambda$, the inclusion relation 
$C_\beta(\mathcal{S}^*(\lambda))\subset \mathcal{S}^*$ holds.
\end{theorem}
\begin{proof}
If $0\leq\lambda<1$ and $0\leq\beta\leq2\lambda$, then we notice that the restrictions 
on the parameters in the hypothesis of Lemma \ref{2lemma2.6} are satisfied. Thus, it gives
$$
C_\beta(\mathcal{S}^*(\lambda))\subset J(\mathcal{S}^*(\lambda-\beta/2)).
$$
Also, if we choose $\alpha=0$ in Lemma \ref{2lemma2.2} we obtain
$$
J(\mathcal{S}^*(\lambda))=\mathcal{K}(\lambda).
$$ 
Combination of the above two relations clearly yields
$$
C_\beta(\mathcal{S}^*(\lambda))\subset \mathcal{K}(\lambda-\beta/2),
$$
which is valid since $\lambda-\beta/2\ge 0$ and we complete the proof.
\end{proof}	

The following example shows that, for $\beta>2\lambda$, the image of 
$\mathcal{S}^*(\lambda)$ under the $\beta$-Ces\`{a}ro transform does not lie in the starlike family.

\begin{example}\label{2example2.11}
Consider the function $f(z)=z/(1-z)^{2-2\lambda}$ for $0\le \lambda<1$. 
The $\beta$-Ces\`aro transform thus takes to the form \eqref{C-beta}.
It is easy to calculate that 
$$
{\rm Re}\,\bigg(1+\frac{zC_\beta[f]''(z)}{C_\beta[f]'(z)}\bigg)
={\rm Re}\,\bigg(1+ \frac{(\beta-2\lambda+2)z}{1-z}\bigg)
>1-\frac{\beta-2\lambda+2}{2}
=\lambda-\frac{\beta}{2},
$$
for $\beta >2\lambda$.

On the other hand, for $2\lambda<\beta$, J. A. Pfaltzgraff, M. O. Reade, 
and T. Umezawa in \cite{Pfaltzgraff76} showed that there exist a point $z_0 \in \mathbb{D}$ such that 
$$
{\rm Re}\,\bigg(\frac{z_0C_\beta[f]'(z_0)}{C_\beta[f](z_0)}\bigg)<0
$$
(see also \cite[pp. 44-45]{MM-Book}).
Hence $C_\beta[f]$ is not a starlike function. 
\end{example}

\begin{remark}
Recall from \cite{Hartmann74} that the Ces\`aro transform does not preserve 
the starlikeness. More generally, here Example \ref{2example2.11} shows that 
the $\beta$-Ces\`{a}ro transform also does not preserve the starlikeness, for any $\beta>0$. 
\end{remark}

We already know that the Alexander transform and the Ces\`{a}ro transform preserve 
the class $\mathcal{K}$. In the following, we determine the values of $\beta$ for 
which the $\beta$-Ces\`aro transform  preserve the class $\mathcal{K}$.

\begin{theorem}\label{2theorem2.13}
For $0\leq\beta\leq 1$, the inclusion relation $C_\beta(\mathcal{K})\subset 
\mathcal{K}((1-\beta)/2)$ holds. In particular, we have $C_\beta(\mathcal{K})\subset \mathcal{K}$.
\end{theorem}
\begin{proof} For $f \in \mathcal{K}$, it is easy to see that
$$
{\rm Re}\,\bigg(1+\frac{zC_\beta[f]''(z)}{C_\beta[f]'(z)}\bigg) 
= {\rm Re}\,\bigg(\frac{zf'(z)}{f(z)}+\frac{\beta z}{1-z}\bigg)>\frac{1-\beta}{2},
$$
since $\mathcal{K}\subset \mathcal{S}^*(1/2)$ and 
${\rm Re}\,(z/(1-z))>-1/2$ for $|z|<1$. By the definition \eqref{2eq1.2}, 
it follows that $C_\beta[f]\in \mathcal{K}((1-\beta)/2)$,  for $0\leq\beta\leq1$. Hence proved. 
\end{proof}	

For $\beta>1$, we have the following counterexample to show that $C_\beta[f]$ 
need not be convex though $f \in \mathcal{K}$. 
\begin{example}
Let $f(z)=z/(1-z)$, $z\in\mathbb{D}$. It is well known that $f\in \mathcal{K}$. We obtain
$$
C_\beta[f](z)=\int_{0}^{z}\frac{1}{(1-w)^{\beta+1}}dw.
$$
It is easy to calculate that 
$$
{\rm Re}\,\bigg(1+\frac{zC_\beta[f]''(z)}{C_\beta[f]'(z)}\bigg)
= 1+{\rm Re}\,\bigg(\frac{(\beta+1)z}{1-z}\bigg).
$$
For $\beta>1$, there is a sequence of points  $z_n=-1+1/n\in\mathbb{D}$
such that
$$
1+{\rm Re}\,\bigg(\frac{(\beta+1)z_n}{1-z_n}\bigg)=\frac{n(1-\beta)+\beta}{2n-1}<0
$$
for $n>\beta/(\beta-1)$.
Therefore, $C_\beta[f]$ need not be a convex function, for $\beta>1$.
\end{example}

The following lemma is due to E. P. Merkes and J. Wright \cite{Merkes71} which 
gives a refinement of Theorem~\ref{2theorem2.13} to the close-to-convex family.
\begin{lemma}\label{2lemma2.15}
Let $f(z)=\sum_{n=1}^{\infty} a_n z^n$ be analytic and $g(z)=\sum_{n=1}^{\infty} b_n z^n$ be an
analytic univalent starlike function in 
$\mathbb{D}$. If $H$ denotes the convex hull of the image of $\mathbb{D}$ under the mapping 
$e^{i\alpha}(f'/g')$ for all $\alpha\in\mathbb{R}$, then $e^{i\alpha}(f/g)\in H$ in $\mathbb{D}$.
\end{lemma}

Now we prove the refinement of Theorem~\ref{2theorem2.13} as indicated above.

\begin{theorem}\label{2theorem2.17}
For $0\leq\beta\leq 1$, the inclusion relation $C_\beta(\mathcal{C})\subset \mathcal{C}$ holds.
\end{theorem}
\begin{proof}
Since $f\in\mathcal{C}$, by its definition there exists a function $\psi\in\mathcal{K}$ 
and $\alpha\in(-\pi/2,\pi/2)$
such that ${\rm Re}\,(e^{i\alpha}f'/\psi')>0$, for $z\in \mathbb{D}$. If $\beta \in [0,1]$, we set
$$
g(z)=\int_{0}^{z} \frac{\psi(w)}{w(1-w)^\beta}dw. 
$$
Then in view of Theorem \ref{2theorem2.13}, $g$ is convex for $0\leq\beta\leq 1$. 
Now we compute and see by using Lemma \ref{2lemma2.15} that
$$
{\rm Re}\,\bigg \lbrace e^{i\alpha}\frac{C_\beta[f]'(z)}{g'(z)}\bigg\rbrace
={\rm Re}\,\bigg\lbrace e^{i\alpha}\frac{f(z)}{\psi(z)}\bigg\rbrace>0
$$
for $z\in\mathbb{D}$. This gives that $C_\beta[f]\in\mathcal{C}$, completing the proof. 
\end{proof}

\begin{remark}
If we choose $\beta>1$ in Theorem~\ref{2theorem2.17}, then the result may not hold 
as can be seen from Example \ref{2example2.8} that the $\beta$-Ces\`aro transform 
$C_\beta$ of the Koebe function is not univalent in $\mathbb{D}$ and hence not close-to-convex.
\end{remark}

Our next result shows that there is a function $f\in \mathcal{S}$ such that 
its $\beta$-Ces\`{a}ro transform is not univalent in $\mathbb{D}$.

\begin{theorem}
There exists a function $f\in \mathcal{S}$ such that $C_\beta[f]$ does not belong 
to $\mathcal{S}$ for $\beta\ge 0$.
\end{theorem}
\begin{proof}
Consider the function $f(z)=z(1-z)^{i-1}$ for $z\in\mathbb{D}$. We can rewrite $f$ 
in the composition form $f=(-i)(g\circ h)$ with
$$
g(z)=z(1-iz)^{i-1}~~\mbox{ and }~~ h(z)=-iz
$$
for $z\in \mathbb{D}$. As shown in \cite[p. 257]{Duren83}, $g$ is univalent in 
$\mathbb{D}$. Since composition of two univalent functions is univalent $f$ is 
univalent in $\mathbb{D}$.
 
Now, if we calculate $C_\beta[f](z)$, for $f(z)=z(1-z)^{i-1}$, then we obtain
$$
C_\beta[f](z)=\int_{0}^{z}(1-w)^{i-1-\beta} dw = \frac{1}{i-\beta}[1-(1-z)^{i-\beta}].
$$
In 1965, W. C. Royster \cite{Royster65} proved that the function 
$g(z)=\exp[\mu \log (1-z)]$ is univalent in $\mathbb{D}$ if and only if $0\neq \mu$ 
lies in either of the closed disks $|\mu+1|\leq 1$, $|\mu-1|\leq 1$. Using this fact, 
we show that $(1-z)^{i-\beta}$ is not univalent in $\mathbb{D}$ for $\beta\neq 1$. 
It thus follows that $C_\beta[f]$ is not univalent in $\mathbb{D}$ for $\beta\neq 1$. 
The remaing case $\beta=1$ is already handled by F. W. Hartmann and T. H. MacGregor 
in \cite{Hartmann74}.	
\end{proof}

\section{Pre-Schwarzian Norms}
Recall the definition of the pre-Schwarzian norm of a function $f\in \mathfrak{F}$:
$$
\|f\|=\sup_{z\in\mathbb{D}}\,(1-|z|^2)\left|\frac{f''(z)}{f'(z)}\right|.
$$
It is well-known that $\|f\|\le 6$ for $f\in\mathcal{S}$ as well as for $f\in\mathcal{S}^*$.
The sharp estimation $\|f\|\le 4$, for $f\in\mathcal{K}$, was later generalized by 
S. Yamashita to the class $\mathcal{K}(\lambda)$, $0\le \lambda<1$ (see \cite{Yamashita99}). 
Recently, in \cite{Ali18}, S. Yamashita's result has been further extended to
$\mathcal{K}(\lambda)$, $-1/2\le \lambda<1$. However, here we prove that the result of 
Yamashita holds true for all $\lambda<1$.

\begin{theorem}\label{2theorem2.18}
For $\lambda<1$, if  $f\in \mathcal{K}(\lambda)$ then $\|f\|\leq 4(1-\lambda)$ and 
the bound is sharp.
\end{theorem}
\begin{proof}
Recall the relation
$$
\mathcal{K}(\lambda)=
(1-\lambda)\star \mathcal{K}=\{(1-\lambda)\star f : f\in \mathcal{K} \}.
$$
Thus, if $f\in \mathcal{K}(\lambda)$, then there exists a function $g\in \mathcal{K}$ such that
$f(z)=(1-\lambda)\star g(z)$. It follows that
$$
\|f\|=(1-\lambda)\|g\|\le 4(1-\lambda),
$$
completing the proof.
\end{proof}

Our purpose in this section is to obtain the pre-Schwarzian norm of the elements 
in $J(\mathcal{S}_{\alpha}^*(\lambda))$, $C_\beta(\mathcal{S}^*(\lambda))$ and in 
$C_\beta(\mathcal{S})$ leading to certain observation highlighted at the end of Section~1.
   
First, as a consequence of Theorem \ref{2theorem2.18}, we obtain a sharp estimate of 
$\|J[f]\|$ for $f\in\mathcal{S}^*_\alpha(\lambda)$. This can be rewritten in the following form.

\begin{theorem}\label{2theorem2.19}
 For each $\alpha\in(-\pi/2,\pi/2)$ and each $\lambda<1$, the sharp inequality 
 $\|f\|\leq 4(1-\lambda)\cos\alpha$ 
 holds for $f\in J(\mathcal{S}_{\alpha}^*(\lambda))$.
\end{theorem}
\begin{proof}
It is easy to calculate that $\|I_\gamma(f)\|=|\gamma|\|f\|$. 
Secondly, By Lemma \ref{2lemma2.2} for $f \in J(\mathcal{S}_{\alpha}^*(\lambda))$ 
there exists a function $k \in \mathcal{K}(\lambda)$ such that 
$f=I_{e^{-i\alpha}\cos\alpha}[k]$. It concludes that 
$\|f\|=|\cos \alpha|\|k\|\leq4(1-\lambda)\cos \alpha$. 

It is evident that the equality holds for the function 
$$
h(z)=\frac{1}{(1-2\lambda)}\Bigg[\frac{1}{(1-z)^{(1-2\lambda)}} - 1\Bigg].
$$
belonging to the class $\mathcal{K}(\lambda)$.
Indeed, we have
$$
\lim_{t\to 1^-} \,(1-t^2)\left|\frac{h''(t)}{h'(t)}\right|
=\lim_{t\to 1^{-}} \,[2(1+t)(1-\lambda)]=4(1-\lambda).
$$
It is easy to compute that the function $g_\alpha\in \mathcal{S}^*_\alpha(\lambda)$ 
corresponding to the function $h(z)$ is given by 
\begin{equation}\label{3eq1}
g_\alpha(z)=z(1-z)^{2(\lambda-1)e^{-i\alpha}\cos\alpha}.
\end{equation}
This completes the proof.
\end{proof}

We remark that if we choose $\lambda=0$ in Theorem \ref{2theorem2.19}, then it 
reduces to  \cite[Proposition~6]{Kim07}.

\begin{remark}
It is well-known that 
for each $\lambda<0$, the class  
$\mathcal{S}^*(\lambda)$ is not contained in the class $\mathcal{S}$ 
(see for instance \cite[p.~66]{MM-Book}). However, here we provide an alternate 
method to show this. As we computed above, 
$\|J[g_0]\|=4(1-\lambda)$ for $g_0\in\mathcal{S}^*(\lambda)$ defined by \eqref{3eq1}. 
For $\lambda<0$, it is clear that $4<\|J[g_0]\|$. Thus, by  
\cite[Theorem 1.1]{Kimponnusamy04}, we conclude that $g_0\not\in\mathcal{S}$. 
\end{remark}

The next theorem obtains the pre-Schwarzian norm estimate of the elements 
in the image set of the $\beta$-Ces\`{a}ro transform of functions from the 
class $\mathcal{S}^*_\alpha(\lambda)$.

\begin{theorem}\label{2theorem2.23}
Let $-\pi/2<\alpha<\pi/2$, $\beta\ge 0$ and $\lambda<1$. 
If $f\in C_\beta(\mathcal{S}^*_\alpha(\lambda))$, then 
the sharp inequality $\|f\|\leq 4(1-\lambda)\cos\alpha+2\beta$ 
holds.
\end{theorem}
\begin{proof}
We observe from the definition of the $\beta$-Ces\`aro transform that the 
inequality $\|C_\beta[f]\|\leq\|J[f]\|+2\beta$ holds for any $f\in\mathfrak{F}$. 
To complete the proof, we recall from Theorem \ref{2theorem2.19} that 
$\|J[f]\|\leq 4(1-\lambda)\cos\alpha$ for $f\in\mathcal{S}^*_\alpha(\lambda)$. It thus concludes that 
$\|C_\beta[f]\|\leq 4(1-\lambda)\cos\alpha+2\beta$ for $f\in \mathcal{S}^*_\alpha(\lambda)$.

For the sharpness, let us consider the function $g_\alpha$ defined by \eqref{3eq1} and we see that 
$$
\lim_{t\to 1^-}(1-t^2)\left|\frac{C_\beta[g_\alpha]''(t)}{C_\beta[g_\alpha]'(t)}\right|
=\lim_{t\to 1^-}(1+t)[2(1-\lambda)\cos\alpha+\beta]
=4(1-\lambda)\cos\alpha+2\beta,
$$
completing the proof.
\end{proof}

A similar technique that is adopted in Theorem~\ref{2theorem2.23} further leads to 
the norm estimate of the $\beta$-Ces\`aro transform of functions $f\in\mathcal{S}$, 
which is presented below.

\begin{theorem}
The sharp inequality $\|f\|\leq 4+2\beta$ 
holds for $f\in C_\beta(\mathcal{S})$.
\end{theorem}
\begin{proof}
As explained in the proof of Theorem~\ref{2theorem2.23}, we have
$\|C_\beta[f]\|\leq\|J[f]\|+2\beta$ for any $f\in \mathfrak{F}$. Also, we recall from  
\cite[Theorem~1.1]{Kimponnusamy04} that $\|J[f]\|\leq 4$ for $f\in \mathcal{S}$. 
Then we conclude that $\|C_\beta[f]\|\leq 4+2\beta$ for $f\in \mathcal{S}$. 

This is sharp as we can see from the sharpness part of the proof of Theorem 
\ref{2theorem2.23} by considering cases $\lambda=0$ and $\alpha=0$ (i.e. by considering 
the Koebe function).
\end{proof}

\end{document}